\documentclass[a4paper, 11pt, reqno]{amsart}
\usepackage[usenames,dvipsnames]{color}
\usepackage{amssymb, amsmath, latexsym, graphics, graphicx}
\newtheorem{theorem}{Theorem}[section]
\newtheorem{lemma}[theorem]{Lemma}
\newtheorem{corollary}[theorem]{Corollary}
\newtheorem{example}[theorem]{Example}
\newtheorem{proposition}[theorem]{Proposition}

\newtheorem{remark}[theorem]{Remark}
\makeatother

\newcommand{\ol}{\overline}

\newcommand{\p}[1]{\langle #1 \rangle}

\def\S{{\mathcal S}}

\def\T{{\mathbb{T}}}
\def\B{{\mathcal B}}

\begin{document}
\title[Free Objects in Triangular Matrix Varieties]{Free Objects in Triangular Matrix Varieties and Quiver Algebras over Semirings}

\date{\today}
\thanks{}

\maketitle

\begin{center}

MARK KAMBITES\footnote{School of Mathematics, University of Manchester,
Manchester M13 9PL, UK. Email \texttt{Mark.Kambites@manchester.ac.uk}.}

\end{center}

\begin{abstract} We study the free objects in the variety of semigroups and 
variety of monoids generated by the monoid of all $n \times n$ upper 
triangular matrices over a commutative semiring. We obtain explicit 
representations of these, as multiplicative subsemigroups of quiver 
algebras over polynomial semirings. In the $2 \times 2$ case this also 
yields a representation as a subsemigroup of a semidirect product of 
commutative monoids. In particular, from the case where $n=2$ and the 
semiring is the tropical semifield, we obtain a representation of the free 
objects in the monoid and semigroup varieties generated by the bicyclic 
monoid (or equivalently, by the free monogenic inverse monoid), inside a 
semidirect product of a commutative monoid acting on a semilattice. We
apply these representations to answer several questions, including that of
when the given varieties are locally finite.
\end{abstract}

There is considerable interest in identities satisfied by matrix semigroups over semirings, due in large part to the potential of semirings as
natural carriers for linear representations of semigroups. A particular focus of recent research has been the case of the \textit{tropical semiring}. Izhakian and Margolis \cite{Izhakian10} initiated work on this topic by providing the first example of a non-trivial semigroup identity satisfied by the semigroup of all $2 \times 2$ tropical matrices; a key step in the proof was the observation that the monoid $UT_2(\mathbb{T})$ of \textit{upper triangular} matrices satisfies an identity. Further work by numerous authors  \cite{Chen16,Daviaud17,Daviaud18,Izhakian14,Izhakian14erratum,Johnson18,Johnson18b,Okninski15,Shitov18,Taylor17} has culminated in the recent proof by Izhakian and Merlet \cite{Izhakian18} (building in particular on ideas of Shitov \cite{Shitov18}) that every semigroup of tropical matrices satisfies some non-trivial identity. Again, the proof relies upon the corresponding statement for upper triangular matrices, which had been previously established \cite{Izhakian14,Izhakian14erratum,Okninski15,Taylor17}. Although the existence of non-trivial identities for tropical matrix semigroups is now resolved, a broader understanding of these identities is still lacking.

The precise identity shown in \cite{Izhakian10} to be satisfied by $UT_2(\mathbb{T})$ is, in fact, a famous example due
to Adjan \cite{Adjan66} of an identity satisfied by the \textit{bicyclic monoid} $\mathcal{B}$. Based on this and the observation that $\mathcal{B}$
has a faithful representation inside $UT_2(\mathbb{T})$, Izhakian and Margolis conjectured that in fact $UT_2(\mathbb{T})$ and $\mathcal{B}$ satisfy
exactly the same identities. Daviaud, Johnson and the author \cite{Daviaud18} developed a systematic way to study identities in upper triangular tropical matrix semigroups, allowing among other things a proof of this conjecture. By Birkhoff's theorem, it follows from this result that
$UT_2(\T)$ and $\B$ generate the same variety of semigroups: we shall call this variety, which also appears as the variety generated by the
\textit{free monogenic inverse monoid} and has been the object of much study, the \textit{bicyclic variety}. 

It is a standard result in universal algebra (see for example \cite[Proposition VI.4.5]{Cohn81}) that every variety contains a (unique up to isomorphism) free object of every rank, but in general the free objects can be hard to describe. At a recent conference\footnote{The 56th Summer School on General Algebra and Ordered Sets (SSAOS2018) in \v{S}pindler\r{u}v Ml\'{y}n, Czechia, September 2018.}, the author was asked whether the methods of \cite{Daviaud18} yield an explicit description of the free object in the bicyclic variety. This paper provides a positive answer to that question, in the form of two closely related explicit descriptions: one (Theorem~\ref{thm_freerep2}) as a multiplicative subsemigroup of a certain quiver algebra over
a tropical polynomial semiring, and the other as a subsemigroup of a semidirect product of a commutative monoid acting on a semilattice (Theorem~\ref{thm_freerep3}).

In fact, our results apply in a broader domain than just the bicyclic variety. Johnson and Fenner \cite{Johnson18} have extended the methods of \cite{Daviaud18} from the tropical case to more general semirings (and also proved some interesting results about upper
\textit{uni}triangular matrix semigroups over semirings, although these will not directly concern us here). Utilising this extension, we can
explicitly describe the free object in the variety generated by $UT_n(\S)$, for any $n$ and $\S$ an arbitrary commutative semiring with $0$ and $1$, as a multiplicative subsemigroup of a certain quiver algebra over a polynomial semiring over $\S$. When $n=2$ we also obtain an embedding
into a semidirect product of commutative semigroups (one of which is a semilattice in the case that the semiring is idempotent). For each result about a variety of semigroups, we also obtain a similar result about the corresponding variety of monoids.

Insofar as we know, our main results are new even when the semiring $\S$ is a field. For a field $k$ of characteristic $0$ the
free objects in the variety generated by $UT_n(k)$ are easily seen to be absolutely free semigroups (or free commutative semigroups if $n=1$), so the results mean
only than that certain subsemigroups of quiver algebras and semidirect products are free. But in fields of prime characteristic, the corresponding free objects seem to be new and potentially interesting examples of semigroups. We suggest that they deserve further study, and that our results will provide a means to this.

In addition to this introduction, this paper comprises four sections. Section~\ref{sec_semiringsquivers} recalls some background material and makes some elementary definitions and observations which, while we do not believe they have appeared in the literature before, are straightforward
extensions of known ideas for rings or fields; these include quiver algebras over semirings. In Section~\ref{sec_freeobj} we show that the approach to understanding identities in triangular matrix semigroups introduced in \cite{Daviaud18} and further developed in \cite{Johnson18} has a natural interpretation in terms of
quiver algebras over semirings, and that this yields a representation of each free object in an upper triangular matrix variety as a multiplicative subsemigroup of a quiver algebra. Section~\ref{sec_reduction} employs a trick to slightly reduce the complexity of this representation;
in the case of $2 \times 2$ matrices (so in particular for the bicyclic variety) this yields a representation within a semidirect
product of commutative monoids. Finally, in Section~\ref{sec_fields} we show that if a semiring $\S$ has the property that distinct formal polynomials in one variable represent distinct functions then $UT_2(\S)$ (and hence also $UT_n(\S)$ and $M_n(\S)$ for all $n \geq 2$) cannot satisfy a semigroup identity; this combines with results of
Okninski and Salwa \cite{Okninski95} to describe exactly for which fields $k$ the semigroup $UT_2(k)$ satisfies an identity or
contains a free subsemigroup.

\section{Semirings, Polynomials and Quivers}\label{sec_semiringsquivers}

In this section we briefly recall some elementary definitions relating to semirings, and proceed to introduce quiver algebras over semirings. To the best of our knowledge these have not previously appeared in the literature, but they are an entirely straightforward generalisation of quiver algebras over fields or rings.

A \textit{semiring} is a set $\S$ endowed with binary operations of addition (denoted by $+$) and multiplication (denoted by $\times$ or just by juxtaposition) and distinguished elements $0 \neq 1$, such that
both operations are associative, addition is commutative, multiplication distributes over addition, $1$ is an identity for multiplication, and $0$ is an identity element for addition and a zero element for multiplication.
The semiring $\S$ is called \textit{commutative} if multiplication is commutative,
and \textit{idempotent} if addition is idempotent, that is, if $a+a = a$ for all $a \in S$. The semiring $\S$ is a \textit{semifield} if $\S \setminus \lbrace 0 \rbrace$ forms an abelian group under multiplication. 

All (commutative) rings with identity (in particular, all fields) are examples of (commutative) semirings. 
A widely studied example with a rather different flavour is the \textit{tropical semifield} $\mathbb{T}$, with elements
$\mathbb{R} \cup \lbrace -\infty \rbrace$, maximum for its addition operation, and classical addition for its multiplication operation.
The two-element subset $\lbrace -\infty, 0 \rbrace \subseteq \mathbb{T}$ is closed under these operations and therefore forms a
\textit{subsemifield}, called the \textit{Boolean semifield} because it is isomorphic to the set $\lbrace \mathrm{False}, \mathrm{True} \rbrace$ under the logical operations
of ``or'' and ``and''. This is one of only two $2$-element semifields (up to isomorphism), the other being of course the field $\mathbb{Z}_2$.
Further examples of semirings can be found in Examples~\ref{example_freeidpt} and \ref{example_interval} below.

Now fix a commutative semiring $\S$ and let $X$ be a finite set of symbols, which we think of as variables. The set of formal polynomials in the
(commuting) variables $X$, with coefficients from $\S$, forms a commutative semiring under the obvious operations, which we denote
$\S[X]$.
Every formal polynomial $f$ naturally defines an \textit{evaluation function} $\overline{f} : \S^X \to \S$.
In general, of course, two distinct formal polynomials may have the same evaluation function; declaring two
formal polynomials to be equivalent if they have the same evaluation function defines a congruence on the semiring
of formal polynomials, the quotient by which equips the set of all polynomial evaluation functions with a commutative semiring structure. We
call this semiring the \textit{polynomial function semiring} over $\S$ in the variables $X$, and denote it $\ol{\S[X]}$. A
formal polynomial is a called a \textit{formal monomial} if it has only one non-zero term; a polynomial function is a \textit{monomial
function} if it is the evaluation function of some formal monomial.

Recall that a \textit{quiver} $\Gamma$ is a directed acyclic (except perhaps for loops) graph, possibly with loops and multiple edges. A \textit{path} $\pi$ in a quiver is a finite sequences of edges satisfying the obvious coterminality conditions; we write $|\pi|$ for the number of edges in $\pi$. For notational purposes, we identify each one-edge path
with its single edge. There is a distinct \textit{empty path} at each vertex $v$. We write $\Gamma^*$ for the set
of all paths in $\Gamma$.

Let $\mathcal{S}$ be a commutative semiring with $0$ and $1$, and $\Gamma$ a quiver. Formally speaking, the \textit{quiver algebra} (or \textit{path algebra}) $\S\Gamma$ is the set of functions from $\Gamma^*$ to $\mathcal{S}$ which are finitely supported (that is, which take all but finitely many paths to
the $0$ element of $\S$), equipped with the operations of pointwise addition and convolution multiplication: 
$$(f+g)(x) = f(x) + g(x) \ \ \textrm{ and } \ \ (fg)(x) = \sum_{st = x} f(s) g(t),$$
where $st$ denotes concatenation of paths, the condition $st = x$ is taken to be false when the product $st$ is not defined, and the empty
sum is the constant $0$ function.
The quiver algebra $\S\Gamma$ is a (non-commutative) semiring with zero element the constant function taking all paths to $0 \in S$,
and identity element the function which takes empty paths to $1 \in \S$ and other paths to $0 \in \S$. 

If we identify each path $x \in \Gamma^*$ with its characteristic function in $\S \Gamma$ then we have
$$f \ = \ \sum_{x \in \Gamma^*} f(x) x$$
for each $f \in \S \Gamma$. In other words,
elements of $\S \Gamma$ can be regarded as formal $\S$-linear combinations of elements of $\Gamma^*$. These expressions can be
manipulated using the usual commutative, associative and distributive laws, with undefined products of terms taken to be $0$ (in other words,
discarded). With this viewpoint, the zero element of $S \Gamma$ is the empty sum and the identity element is the sum of all the
empty paths at different vertices in $\Gamma$.

Throughout this paper, $\mathcal{S}$ will denote a commutative semiring with identity $1$ and zero $0$. For $n$ a natural number we
write $[n]$ for the set $\lbrace 1, \dots, n \rbrace$.

\section{Free objects as subsemigroups of quiver algebras}\label{sec_freeobj}

In this section we fix a commutative semiring $\S$ and show that the free objects in the variety generated by $UT_n(\S)$ arise as multiplicative subsemigroups of certain quiver algebras over polynomial semirings over $\S$. To this end, we begin by reformulating some definitions and results from \cite{Daviaud18} and \cite{Johnson18} in a language based on quivers. 

Fix a natural number $n$ and let $\Sigma$ be a (not necessarily finite) alphabet.
For each $v \in [n]$ and $\sigma \in \Sigma$, let $\sigma_v$ be a formal variable, and define an alphabet
$$X_{n,\Sigma} = \lbrace \sigma_v \mid \sigma \in \Sigma, v \in [n] \rbrace.$$
Recall from Section~\ref{sec_semiringsquivers} that 
$\ol{\S[X_{n,\Sigma}]}$ denotes the semiring of polynomial \textit{functions} (as distinct from formal polynomials) over the variable set $X_{n,\Sigma}$.

We denote by $\Gamma^0_{n,\Sigma}$ the quiver
which has vertex set $[n] = \lbrace 1, \dots, n \rbrace$ and for each $i \leq j$, $|\Sigma|$ edges from $i$ to $j$, labelled by the
letters in $\Sigma$. For notational convenience we identify the edge set with the set of triples
$$\lbrace (i, \sigma, j) \mid i \leq j, \sigma \in \Sigma \rbrace$$
in the obvious way.
Note that this quiver has both multiple edges and loops. We write $\Gamma_{n,\Sigma}$ for the quiver obtained from
$\Gamma^0_{n,\Sigma}$ by
removing the loops. In particular, notice that $\Gamma_{2,\Sigma}$ is the $|\Sigma|$-arrowed \textit{Kronecker quiver}.
In both $\Gamma^0_{n,\Sigma}$ and $\Gamma_{n,\Sigma}$, the labelling of edges by letters in $\Sigma$ extends naturally to
a labelling of paths by words in $\Sigma^*$. We shall write paths in $\Gamma^0_{n,\Sigma}$ and $\Gamma_{n,\Sigma}$ in angle
brackets, using notation such as $\p{1 \xrightarrow{a} 2 \xrightarrow{b} 2 \xrightarrow{b} 4}$ with the obvious meaning. In particular,
the empty path at the vertex $2$, for example, is simply written $\p{2}$.

We are especially interested in the quiver algebra $\ol{\S[X_{n,\Sigma}]} \Gamma_{n,\Sigma}$.

Now let $w$ be a word in $\Sigma^*$ and $\pi$ a path of length $k$ in $\Gamma_{n,\Sigma}$.  We define a \textit{$\pi$-$w$-amble} to be a path
in $\Gamma^0_{n,\Sigma}$ with label $w$, and with the property that removing the loop edges yields the path $\pi$. Notice that each
$\pi$-$w$-amble $\tau$ naturally corresponds to an occurrence of the label of $\pi$ as a scattered subword in $w$, namely the occurrence comprising those letters of $w$ which are read along non-loop edges in $\tau$. 
Notice also that $w$ and $\pi$ are both recoverable from $\tau$ (as
its label, and the path obtained by removing loops, respectively).

Now fix a $\pi$-$w$-amble $\tau$, and for $i \in [|w|] = [|\tau|]$ define
$$x_i = \begin{cases}
1 \in \S & \textrm{ if the $i$th edge of $\tau$ is a non-loop; or} \\
\sigma_v \in X_{n,\Sigma} & \textrm{ if the $i$th edge of $\tau$ is a loop at vertex $v$ with label $\sigma$.}
\end{cases}
$$
 We associate to $\tau$ a 
monomial function
$$\mu_\tau \ = \ \prod_{i = 1}^{|\tau|} x_i  \ \in \ \ol{\S [X_{n,\Sigma}]}$$
which we call the \textit{corresponding monomial} to $\tau$. Here the empty product is taken to be the identity element $1$, so if $w$ is the empty word
and $\tau$ therefore an empty path we have $\mu_\tau = 1$.

We associate to $w$ and $\pi$ the polynomial function
$$f_\pi^w \ = \ \sum_{\tau} \mu_\tau \ \in \ \ol{\S [X_{n,\Sigma}]}$$
where the summation index $\tau$ ranges over $\pi$-$w$-ambles. The empty sum is taken to be the zero element $0$, so if
there are no $\pi$-$w$-ambles (that is, if the label of $\pi$ does not occur as a scattered subword in $w$) then $f_\pi^w = 0$.

\begin{remark}\label{remark_dictionary}
The functions $f_\pi^w$ are an alternative notation for certain functions first defined in the tropical case by Daviaud, Johnson and the
author \cite{Daviaud18} and extended to the general commutative semiring case by Johnson and Fenner \cite{Johnson18}.
A path $\pi$ in $\Gamma_{n,\Sigma}$ corresponds, in the language of \cite{Johnson18}, to a $\rho$ in $[n]$ which is the sequence of vertices
together with a word $u \in \Sigma^*$ which is its label. Each of our variables $\sigma_i \in X_{n,\Sigma}$ corresponds to the variable
which in \cite{Johnson18} is called $x(\sigma,i)$. Through this dictionary, for each $w \in \Sigma^*$, the definition of our function $f_\pi^w$ is the same as the definition of the function called $f_{u,\rho}^w$ in \cite{Johnson18}. The functions which in \cite{Johnson18} are called $f_u^w$ are those
corresponding to $f_\pi^w$ where $\pi$ is the unique path from $1$ to $|u|+1$ labelled $u$ in $\Gamma_{n,\Sigma}$.
\end{remark}

\begin{example}
Let $n = 3$, consider the word $w = abba$, and the path $\pi = \langle 1 \xrightarrow{a} 2 \xrightarrow{b} 3 \rangle$. There are two $\pi$-$w$-ambles
corresponding to the two ways in which $ab$ occurs as a scattered subword in $abba$: they are
$\langle 1 \xrightarrow{a} 2 \xrightarrow{b} 2 \xrightarrow{b} 3 \xrightarrow{a} 3 \rangle$ and
$\langle 1 \xrightarrow{a} 2 \xrightarrow{b} 3 \xrightarrow{b} 3 \xrightarrow{a} 3 \rangle$. The corresponding monomials are $b_2 a_3$ and $b_3 a_3$ respectively, so $f_{\pi}^w = b_2 a_3 + b_3 a_3$.
\end{example}

\begin{remark}\label{remark_emptypath}
Notice that if $\pi$ is the empty path at some vertex $v$ then there is only a single $\pi$-$w$-amble (namely, the path in
$\Gamma_{n,\Sigma}$ which reads $w$ round the loops at $v$) and $f_\pi^w$ is therefore a monomial. In fact in this case $f_\pi^w$ is essentially
just the abelianisation of the word $w$, viewed through a correspondence which identifies $\sigma \in \Sigma$ with $\sigma_v \in X_{n,\Sigma}$.
\end{remark}

\begin{remark}\label{remark_homogeneous}
For a fixed $\pi$, $w$ and a fixed $\sigma \in \Sigma$, every term in the definition of $f_{\pi}^w$ contains the same number of variables of the form $\sigma_i$ (totalled for all values of $i$): namely, the number of times $\sigma$ appears in $w$ minus the number of times it appears in the label of $\pi$. Of course if $\sigma$ appears fewer times in $w$ than in the label of $\pi$, there are no $\pi$-$w$-ambles so $f_{\pi}^w = 0$.
\end{remark}

The idea behind the following remark is implicit in the proof of \cite[Theorem 2.2]{Johnson18}; it is helpful for our purposes to make it explicit.

\begin{remark}\label{remark_relabel}
Suppose $\pi$ and $\phi$ are paths in $\Gamma_{n,\Sigma}$ with the same label (and hence the same length). Choose a permutation $\eta$ of $[n]$
which maps the $k$th vertex of $\pi$ to the $k$th vertex of $\phi$ for every $k \in [|\pi|] = [|\phi|]$, and consider the permutation of $X_{n,\Sigma}$ which
takes $\sigma_i$ to $\sigma_{\eta(i)}$ for each $\sigma \in \Sigma$ and $i \in [n]$. Since the definition of the polynomial function
semiring $\ol{\S[X_{n,\Sigma}]}$ is symmetric in the different variables, this permutation
extends to a semiring automorphism of $\ol{\S[X_{n,\Sigma}]}$, and it is easy to see from the definitions that this automorphism maps 
$f_\pi^w$ to $f_\phi^w$ for all $w \in \Sigma^*$. The existence of such an automorphism means in particular that, for every $w,v \in \Sigma^*$, we have $f_\pi^w = f_\pi^v$ if and only $f_\phi^w = f_\phi^v$.
\end{remark}

The following theorem is essentially a restatement (with different notation) of a result which for the tropical semiring was established by Daviaud, Johnson and the author \cite{Daviaud18}, and for general commutative semirings is implicit in recent work of Johnson and Fenner \cite{Johnson18}.

\begin{theorem}\label{thm_fenner}
Let $\S$ be a commutative semiring with $0$ and $1$, $\Sigma$ an alphabet and $v, w \in \Sigma^*$. Then the identity $v = w$ holds in $UT_n(\S)$ if and only if $f_\pi^v$ = $f_\pi^w$ for all paths in $\pi$ in $\Gamma_{n,\Sigma}$.
\end{theorem}

\begin{proof}
If $f_\pi^v = f_\pi^w$ (in our notation) for all paths $\pi$ in $\Gamma_{n,\Sigma}$ then in particular (through the dictionary in Remark~\ref{remark_dictionary}) in the notation of \cite{Johnson18} we have $f_{u}^w = f_{u}^v$ for
all words $u$ of length $n-1$ or less, so by \cite[Theorem 2.2]{Johnson18} the identity $v=w$ holds in $UT_n(\S)$.

Conversely, if the identity $v=w$ holds in $UT_n(\S)$ then again using \cite[Theorem 2.2]{Johnson18} and Remark~\ref{remark_dictionary} for every word $u \in \Sigma^*$ of length $n-1$ or less
there is a path $\pi$ in $\Gamma_{n,\Sigma}$ with $f_\pi^v = f_\pi^w$. It follows by Remark~\ref{remark_relabel} that we have $f_\pi^v = f_\pi^w$
for all paths $\pi$ in $\Gamma_{n,\Sigma}$. 
\end{proof}
\begin{remark}
The statement of \cite[Theorem 2.2]{Johnson18} is deliberately formulated to reduce the number of polynomials considered, using the principle of
Remark~\ref{remark_relabel}. This optimisation is conducive to efficiency in algorithmic applications, but destroys a symmetry
which is helpful for our current, more theoretical, approach. The role of Remark~\ref{remark_relabel} in the above proof of our Theorem~\ref{thm_fenner} is to ``undo'' this optimisation, allowing us to obtain an ``unoptimised'' result with more symmetry. One
could alternatively establish Theorem~\ref{thm_fenner} directly using the kind of arguments used to prove \cite[Theorem 5.2]{Daviaud18}; for this approach the observation in Remark~\ref{remark_relabel} is not needed.
\end{remark}

Theorem~\ref{thm_fenner} has an interpretation in a certain quiver algebra over the semiring of
polynomial functions over $\S$. We define a function
$$\rho \ : \ \Sigma^* \to \ol{\S[X_{n,\Sigma}]} \Gamma_{n,\Sigma}, \ \ w \mapsto \sum_{\pi} f_{\pi}^w \pi$$
where the summation index $\pi$ ranges over all paths in $\Gamma_{n,\Sigma}$. We note that for $\sigma \in \Sigma$ the image
$\rho(\sigma)$ is the sum of all the length-$1$ paths labelled $\sigma$ and for each $v \in [n]$, $\sigma_v$ times the empty
path at $v$.

\begin{example}\label{example_rho}
Suppose $n = 3$ and $a, b \in \Sigma$ are generators. Then
$$\rho(a) \ = \ a_1 \langle 1 \rangle + a_2 \langle 2 \rangle + a_3 \p{3} + \langle 1 \xrightarrow{a} 2 \rangle + \langle 2 \xrightarrow{a} 3 \rangle + \langle 1 \xrightarrow{a} 3 \rangle, \textrm{ and }$$
$$\rho(b) \ = \ b_1 \p{1} + b_2 \p{2} + b_3 \p{3} + \p{1 \xrightarrow{b} 2} + \p{2 \xrightarrow{b} 3} + \p{1 \xrightarrow{b} 3}$$
while
\begin{align*}
\rho(ab) \ = \ &a_1 b_1 \p{1} + a_2 b_2 \p{2} + a_3 b_3 \p{3} \\
&+ b_2 \p{1 \xrightarrow{a} 2} + b_3 \p{1 \xrightarrow{a} 3} + b_3 \p{2 \xrightarrow{a} 3} \\
&+ a_1 \p{1 \xrightarrow{b} 2} + a_1 \p{1 \xrightarrow{b} 3} + a_2 \p{2 \xrightarrow{b} 3} \\
&+ \p{1 \xrightarrow{a} 2 \xrightarrow{b} 3}.
\end{align*}
where the terms in the four lines of the latter correspond respectively to the four scattered subwords of $ab$ (respectively $\epsilon$, $a$, $b$ and $ab$). Notice that the
coefficients of are all monomials; this reflects the fact that in these examples each scattered subword of $ab$ occurs as such in a unique way; in other
words, for each path $\pi$ there is at most one $\pi$-$ab$-amble. In contrast,
\begin{align*}
\rho(aa) \ = \ &a_1^2 \p{1} + a_2^2 \p{2} + a_3^2 \p{3} \\
&+ (a_1 + a_2) \p{1 \xrightarrow{a} 2} + (a_1 + a_3) \p{1 \xrightarrow{a} 3} + (a_2 + a_3) \p{2 \xrightarrow{a} 3} \\
&+ \p{1 \xrightarrow{a} 2 \xrightarrow{a} 3}.
\end{align*}
where the three lines correspond to the scattered subwords $\epsilon$, $a$ and $aa$, and the binomial terms occurring in the second
line reflect the fact that $a$ occurs twice as a scattered subword of $aa$.

Notice that if we compute products in the quiver algebra, it transpires that $\rho(a) \rho(b) = \rho(ab)$, and similarly $\rho(a) \rho(a) = \rho(aa)$.
As we shall see shortly (in Lemma~\ref{lemma_morphism} and Theorem~\ref{thm_freerep1}), this is not a coincidence.
\end{example}

Since the different functions $f_{\pi}^w$ are clearly
recoverable from $\rho(w)$, Theorem~\ref{thm_fenner} can be restated as follows:
\begin{theorem}\label{thm_identities}
Let $\S$ be a commutative semiring with $0$ and $1$, $\Sigma$ an alphabet and $u, v \in \Sigma^*$, and let $\rho$ be as defined above. Then the identity $u = v$ holds in $UT_n(\S)$ if and only if $\rho(u) = \rho(v)$ in $\ol{\S[X_{n,\Sigma}]} \Gamma_{n,\Sigma}$.
\end{theorem}
Taken alone, Theorem~\ref{thm_identities} is simply a restatement in an alternative notation of Theorem~\ref{thm_fenner}, which itself is
essentially a reformulation of results from \cite{Daviaud18} and \cite{Johnson18}. However, the key observation of this paper is that the function $\rho$ is actually a \textit{morphism} of monoids from $\Sigma^*$ to the multiplicative monoid of the
quiver algebra $\ol{\S[X_{n,\Sigma}]} \Gamma_{n,\Sigma}$. From this, it follows that its image, as a multiplicative semigroup, is a faithful, concrete representation of the free
$\Sigma$-generated object in the monoid variety generated by $UT_n(\S)$. The rest of this section is concerned with establishing
these claims.

\begin{lemma}\label{lemma_splitpath}
Let $u, v \in \Sigma^*$, $\pi$ be a path in $\Gamma_{n,\Sigma}$, and $m$ a monomial function in $\ol{\S [X_{n,\Sigma}]}$. Then there is a one-to-one correspondence between
\begin{itemize}
\item $\pi$-$uv$-ambles with corresponding monomial $m$; and
\item quadruples $(\alpha, \beta, \gamma, \delta)$ where $\pi = \alpha \beta$ is a factorisation of $\pi$, $\gamma$ is an $\alpha$-$u$-amble
and $\delta$ is a $\beta$-$v$-amble such that $\mu_\gamma \mu_\delta = m$.
\end{itemize}
\end{lemma}
\begin{proof}
Suppose $\tau$ is a $\pi$-$uv$-amble with corresponding monomial $m$. Let $\gamma$ be the prefix of $\tau$ labelled by
$u$ and $\delta$ the suffix of $\tau$ labelled by $v$. Let $\alpha$ and $\beta$ be the paths in $\Gamma_{n,\Sigma}$ obtained by removing
the loop edges from $\gamma$ and $\delta$ respectively. It is immediate from the definitions that
\begin{itemize}
\item $\pi = \alpha \beta$;
\item $\gamma$ is an $\alpha$-$u$-amble; and
\item $\delta$ is a $\beta$-$v$-amble.
\end{itemize}
Since the paths $\gamma$ and $\delta$ between them trace exactly the same loop edges as $\tau$, we see that
$\mu_\gamma \mu_\delta = \mu_\tau = m$ by definition as formal monomials, and therefore also as monomial functions. Thus, $(\alpha, \beta, \gamma, \delta)$ is a quadruple with the given
properties.

It is clear that this quadruple is uniquely determined by the $\pi$-$uv$-amble $\tau$. Moreover, every quadruple $(\alpha, \beta, \gamma, \delta)$
with the given properties will arise in this way by considering the $\pi$-$uv$-amble $\gamma \delta$.
\end{proof}

\begin{lemma}\label{lemma_morphism}
With notation as above, the map $\rho$ is a monoid morphism from the free monoid $\Sigma^*$ to the multiplicative semigroup of
the quiver algebra $\ol{\S[X_{n,\Sigma}]} \Gamma_{n,\Sigma}$.
\end{lemma}

\begin{proof}
First let $\epsilon$ denote the empty word (which is the identity element in $\Sigma^*$) and $\pi$ be a path in $\Gamma_{n,\Sigma}$.
Notice that $\pi$-$\epsilon$-ambles exist only in the case that $\pi$ is the empty path at some vertex $v$, in which case the unique $\pi$-$\epsilon$-amble is itself the empty path at $v$. It follows that $f_\pi^\epsilon$ is  $1$ if $\pi$ is an empty path and $0$ otherwise, so $\rho(\epsilon)$ by definition is the sum of the empty paths in $\Gamma_{n,\Sigma}$, which is the identity element of the quiver algebra $\ol{\S[X_{n,\Sigma}]} \Gamma_{n,\Sigma}$.

Now let $u, v \in \Sigma^*$ and consider a path $\pi$ in $\Gamma_{n, \Sigma}$. We must show that the coefficient of $\pi$
in $\rho(uv)$ (which by definition is $f_\pi^{uv}$) is the same as the coefficient of $\pi$ in $\rho(u) \rho(v)$. By definition of the multiplication in the quiver algebra, the latter is the sum over all factorisations $\pi = \alpha \beta$ of the coefficient of $\alpha$ in $\rho(u)$ (which by
definition is $f_{\alpha}^u$) and the coefficient of $\beta$ in $\rho(v)$ (which by definition is $f_{\beta}^v)$. Thus, it will suffice to show that
\begin{equation}\label{eq}
f_{\pi}^{uv} \ = \ \sum_{\pi = \alpha \beta} f_{\alpha}^u f_{\beta}^v
\end{equation}
as functions in $\ol{\S[X_{n,\Sigma}]} \Gamma_{n,\Sigma}$. In fact we shall show the sides of $\eqref{eq}$ are equal even as formal polynomials
in $\S[X_{n,\Sigma}] \Gamma_{n,\Sigma}$.

Consider the equation $\eqref{eq}$ with both sides expanded using the definitions of $f_{\pi}^{uv}$, $f_{\alpha}^u$ and $f_{\beta}^v$.
Consider a monomial $m$ in $\S [ X_{n,\Gamma}]$. Each occurrence of $m$ as a term on the left-hand-side of \eqref{eq} comes from a $\pi$-$uv$-amble with corresponding monomial $m$. Each occurrence of $m$ on the right-hand-side
corresponds to a factorisation $\pi = \alpha \beta$, an $\alpha$-$u$-amble $\gamma$ and a $\beta$-$v$-amble $\delta$ such that
$\mu_\gamma \mu_\delta = m$. By Lemma~\ref{lemma_splitpath}, it follows that $m$ occurs as a term the same (possibly $0$) number
of times on each side. Since this holds for all monomials $m$, we deduce that the left-hand-side and right-hand-side are the same
formal polynomial, and therefore in particular represent the same function.
\end{proof}

We are now ready to prove the main theorem of this section.

\begin{theorem}\label{thm_freerep1}
Let $\S$ be a commutative semiring with $0$ and $1$, let $n$ be a natural number and $\Sigma$ a (not necessarily finite)
alphabet. Let $\rho$ be defined as above. Then $\rho(\Sigma^*)$ [respectively, $\rho(\Sigma^+)$] is a multiplicative submonoid
[subsemigroup] of $\ol{\S[X_{n,\Sigma}]} \Gamma_{n,\Sigma}$ which is free of
rank $|\Sigma|$ (freely generated by the elements $\rho(\sigma)$ for $\sigma \in \Sigma$) in the monoid
variety [semigroup variety] generated by $UT_n(\S)$.
\end{theorem}
\begin{proof}
It follows from Lemma~\ref{lemma_morphism} that $\rho(\Sigma^*)$ is a submonoid generated by the elements $\rho(\sigma)$
for $\sigma \in \Sigma$. Moreover, it is easily seen that the identity element in a quiver algebra is multiplicatively indecomposable,
from which it follows easily that $\rho(\Sigma^+)$ is a subsemigroup. It remains to show that this monoid [semigroup] lies in the
given monoid variety [semigroup variety], and that it has the universal property defining a free object in that variety.

By Birkhoff's theorem \cite{Birkhoff35}, to show that $\rho(\Sigma^*)$ lies in the given monoid variety, it suffices show that it satisfies all the identities which hold in $UT_n(\S)$. Suppose, then, that a monoid identity $u = v$ holds in $UT_n(\S)$.  We must show that that the identity holds for every
assignment of values in $\rho(\Sigma^*)$ to the variables (not just the assignment given by $\rho$ itself), that is, that for every morphism
$g : \Sigma^* \to \rho(\Sigma^*)$ we have $g(u) = g(v)$. Suppose then that $g$ is such a map. For each $\sigma \in \Sigma$, since
$g(\sigma)$ lies in $\rho(\Sigma^*)$ we may choose $\hat\sigma \in \Sigma^*$ such that $g(\sigma) = \rho(\hat\sigma)$. Since $\Sigma^*$
is a free monoid, the map
$\sigma \to \hat\sigma$ extends to an endomorphism of $\Sigma^*$, taking each $w$ to the word $\hat{w}$ obtained by substituting
$\hat\sigma$ for each letter $\sigma$. Since $g$ (by definition) and $\rho$ (by Lemma~\ref{lemma_morphism}) are both morphisms we have $g(w) = \rho(\hat{w})$ for all words $w$.
Notice that the identity $\hat{u} = \hat{v}$ holds in $UT_n(\S)$ (because it is obtained from $u=v$ by applying a consistent substitution for the letters). Hence, by Theorem~\ref{thm_identities}, $f_{\pi}^{\hat{u}} = f_{\pi}^{\hat{v}}$ for all paths
$\pi$ in $\Gamma_{n,\Sigma}$, which by the definition of $\rho$ means that $\rho(\hat{u}) = \rho(\hat{v})$. But now we have
$$g(u) \ = \ \rho(\hat{u}) \ = \ \rho(\hat{v}) \ = \ g(v).$$
as required.

Since every semigroup identity satisfied by $UT_n(\S)$ is also a monoid identity satisfied by $UT_n(\S)$, and since $\rho(\Sigma^+)$ is a subsemigroup of $\rho(\Sigma^*)$, it follows also that $\rho(\Sigma^+)$ lies in the semigroup variety generated by $UT_n(\S)$.

Next we show that $\rho(\Sigma^*)$ has the universal property of a free object (freely generated by $\rho(\Sigma)$) in the monoid variety generated by $UT_n(\S)$. Let $M$ be a monoid in the given variety and $f : \rho(\Sigma) \to M$ a function. We must show that $f$
extends uniquely to a monoid morphism from $\rho(\Sigma^*)$ to $M$.
Since $\Sigma^*$ is an (absolutely) free monoid on $\Sigma$, there is a (unique) monoid morphism $g : \Sigma^* \to M$ satisfying
$g(\sigma) = f(\rho(\sigma))$ for all $\sigma \in \Sigma$. Now if $\rho(u) = \rho(v)$ then by Theorem~\ref{thm_identities} the identity
$u=v$ is satisfied in $UT_n(\S)$ and therefore also in $M$. It follows that there is a well-defined morphism extending $f$ to $\rho(\Sigma^*)$:
$$\rho(\Sigma^*) \to M, \ \ \rho(w) \mapsto g(w).$$
That this is the unique such extension follows from the fact that $\rho(\Sigma^*)$ is generated by $\rho(\Sigma)$.

Finally, a near-identical argument (substituting ``semigroup'' for ``monoid'' and ``$\Sigma^+$'' for ``$\Sigma^*$'') shows that $\rho(\Sigma^+)$
is a free object in the semigroup variety generated by $UT_n(\S)$.
\end{proof}

\section{Variable Reduction and the $2 \times 2$ Case}\label{sec_reduction}

In this section we employ an idea, the essence of which is implicit in \cite{Daviaud18}, to slightly reduce the number of variables required in the polynomial semiring over which we construct the quiver algebra in which we represent our semigroups of interest. Specifically, we show that the free object in the semigroup variety generated by $UT_n(\S)$ arises as a multiplicative subsemigroup of
$\ol{\S[X_{n-1,\Sigma}]} \Gamma_{n,\Sigma}$ (as opposed to $\ol{\S[X_{n,\Sigma}]} \Gamma_{n,\Sigma}$ as given by Theorem~\ref{thm_freerep1}).
For general $n$ this reduction is probably of limited use, but in the case $n=2$ it significantly simplifies things and yields representations of the free objects in the relevant varieties inside semidirect products of commutative monoids. In particular, this leads to a description of each free object in the bicyclic variety inside a semidirect product of a commutative monoid acting on a semilattice.

Let $\delta : \ol{\S[X_{n,\Sigma}]} \to \ol{\S[X_{n-1,\Sigma}]}$ be the partial evaluation map which evaluates $\sigma_n$ to $1 \in \S$ for each $\sigma \in \Sigma$. We define a map of quiver algebras
$$\lambda : \ol{\S[X_{n,\Sigma}]} \Gamma_{n,\Sigma} \to \ol{\S[X_{n-1,\Sigma}]} \Gamma_{n,\Sigma}, \ \ \textrm{ by } (\lambda(f))(\pi)= \delta(f(\pi)).$$
In other words, $\lambda$ applies the partial evaluation map $\delta$ to each coefficient. It is easy to see that $\delta$ is a semiring morphism,
and from this it is straightforward to deduce that $\lambda$ is a semiring morphism (and in particular, a morphism of multiplicative monoids).

\begin{example}\label{example_lambdarho}
Continuing from Example~\ref{example_lambdarho}, with $n=3$ and $a, b \in \Sigma$:
$$\lambda(\rho(a)) \ = \ a_1 \langle 1 \rangle + a_2 \langle 2 \rangle + \p{3} + \langle 1 \xrightarrow{a} 2 \rangle + \langle 2 \xrightarrow{a} 3 \rangle + \langle 1 \xrightarrow{a} 3 \rangle,$$
$$\lambda(\rho(b)) \ = \ b_1 \p{1} + b_2 \p{2} + \p{3} + \p{1 \xrightarrow{b} 2} + \p{2 \xrightarrow{b} 3} + \p{1 \xrightarrow{b} 3},$$
\begin{align*}
\lambda(\rho(ab)) \ = \ &a_1 b_1 \p{1} + a_2 b_2 \p{2} + \p{3} \\
&+ b_2 \p{1 \xrightarrow{a} 2} + \p{1 \xrightarrow{a} 3} + \p{2 \xrightarrow{a} 3} \\
&+ a_1 \p{1 \xrightarrow{b} 2} + a_1 \p{1 \xrightarrow{b} 3} + a_2 \p{2 \xrightarrow{b} 3} \\
&+ \p{1 \xrightarrow{a} 2 \xrightarrow{b} 3}.
\end{align*}
and
\begin{align*}
\lambda(\rho(aa)) \ = \ &a_1^2 \p{1} + a_2^2 \p{2} + \p{3} \\
&+ (a_1 + a_2) \p{1 \xrightarrow{a} 2} + a_1 \p{1 \xrightarrow{a} 3} + a_2 \p{2 \xrightarrow{a} 3} \\
&+ \p{1 \xrightarrow{a} 2 \xrightarrow{a} 3}.
\end{align*}
\end{example}

\begin{theorem}\label{thm_freerep2}
Let $n \geq 2$ and 
define $\rho$ as in the previous section. Then $\lambda(\rho(\Sigma^*))$ [respectively $\lambda(\rho(\Sigma^+))]$ under the 
multiplication in $\ol{\S[X_{n-1,\Sigma}]} \Gamma_{n,\Sigma}$ is a submonoid [subsemigroup] which is free of
rank $|\Sigma|$ (freely generated by the elements $\lambda(\rho(\sigma))$ for $\sigma \in \Sigma$) in the monoid variety [semigroup
variety] generated by $UT_n(\S)$.
\end{theorem}
\begin{proof}
In view of Theorem~\ref{thm_freerep1} and the fact that $\lambda$ is a morphism of monoids, it suffices to show that $\lambda$ is injective when restricted to the image $\rho(\Sigma^*)$ of $\rho$. We shall see that this follows from Remark~\ref{remark_homogeneous}.

Indeed, suppose $w \in \Sigma^*$ and consider $\rho(w)$. By definition the coefficient in $\rho(w)$ of each path $\pi$ is $f_\pi^w$. The map $\lambda$ applies the map $\delta$ to each coefficient, that is, removes variables $\sigma_n$ for different $\sigma$. By looking at $\lambda(\rho(w))$
we can therefore recover $\delta(f_\pi^w)$ for each path $\pi$. To show the map is injective, it will suffice to show that we can recover the polynomials
$f_\pi^w$ from their images $\delta(f_\pi^w)$.

Consider the empty path $\p{1}$ at the vertex $1$. Since this path does not visit the vertex $n$, the polynomial $f_{\p{1}}^w$ contains no variables of the form
$\sigma_n$, so $f_{\p{1}}^w = \delta(f_{\p{1}}^w)$. 

By Remark~\ref{remark_homogeneous}, knowing $f_{\p{1}}^w$ is enough to tell us for each $\sigma$ how many times each letter occurs in $w$, and hence (using Remark~\ref{remark_homogeneous} again) how many times variables of the form $\sigma_i$ (for different $i$) occur in each term of $f_\pi^w$. Since we know
$\delta(f_\pi^w)$, and know that it has been obtained from $f_\pi^w$ by removing only variables of
the form $\sigma_n$, there is a unique way to recover $f_\pi^w$ by inserting into each term the appropriate number of $\sigma_n$s for each $\sigma$.
\end{proof}

The above trick is particularly useful in the case $n=2$, where it allows us to reduce consideration
to an alphabet $X_{1,\Sigma}$ in one-to-one correspondence with $\Sigma$ itself. This will allow us to 
exhibit an alternative representation within a semidirect product of commutative semigroups. For notational
convenience, for the rest of this section we shall identify
each variable $\sigma_1 \in X_{1,\Sigma}$ with the symbol $\sigma$ itself, so that $X_{1,\Sigma}  = \Sigma$.

Let $A$ denote the set of monomial functions in $\ol{\S[\Sigma]}$, considered as a semigroup under multiplication. Let
$B$ be the whole of $\ol{\S[\Sigma]}$ considered as a semigroup under addition, and let $B^\Sigma$ be the direct
product of $|\Sigma|$ copies of $B$. Notice that $A$ and $B$ are both commutative semigroups. In the case that $\S$ is
an idempotent semiring, $B$ is a commutative semigroup
of idempotents (a \textit{semilattice} in the usual terminology of semigroup theory).
There is a natural action of $A$ on the left (say) of $B = \ol{\S[\Sigma]}$ by the multiplication in $\ol{\S[\Sigma]}$, which extends componentwise to an
action of $A$ on $B^\Sigma$.

Let $G$ be the obvious semidirect product of $B^\Sigma$ with $A$, that is, the set $B^\Sigma \times A$ with
multiplication given by
$$(f,b) (g,c) = (f+bg,bc).$$
Then $G$ is a monoid with identity element $(0, 1)$ where $0$ here represents the identity element of the additive monoid $B^\Sigma$,
which is the constant function taking everything in $\Sigma$ to the constant zero function in $\ol{\S[\Sigma]}$.
Define a map from $\alpha : \ol{\S[\Sigma]} \Gamma_{2,\Sigma} \to G$ by $\alpha(p) = (b_p,a_p)$ where
\begin{itemize}
\item $b_p(\sigma) = p(\p{1 \xrightarrow{\sigma} 2})$ (the coefficient in $p$ of the unique path from $1$ to $2$ labelled $\sigma$ in $\Gamma_{2,\Sigma}$); and
\item $a_p = p(\p{1})$ (the coefficient in $p$ of the empty path at the vertex $1$). By Remark~\ref{remark_emptypath},
this coefficient is a monomial in the variables $\sigma_1$ for $\sigma \in \Sigma$, which here we are identifying
with the symbols in $\Sigma$ itself, so that $a$ does indeed lie in $A$ as required.
\end{itemize}
We shall continue to use the notation $b_p$ and $a_p$ for the first and second components respectively of $\alpha(p)$.

\begin{example}\label{example_alphalambdarho}
Calculating as in Examples ~\ref{example_rho} and \ref{example_lambdarho} but with $n=2$ and using
the convention that each $\sigma \in \Sigma$ is identified with the variable $\sigma_1$ we have
$$\lambda(\rho(a)) \ = \ a \langle 1 \rangle + \langle 2 \rangle + \langle 1 \xrightarrow{a} 2 \rangle$$
$$\lambda(\rho(b)) \ = \ b \p{1} + \p{2} + \p{1 \xrightarrow{b} 2},$$
$$\lambda(\rho(ab)) \ = \ a b \p{1} + \p{2} + \p{1 \xrightarrow{a} 2} + a \p{1 \xrightarrow{b} 2}$$
$$\lambda(\rho(aa)) \ = \ a^2 \p{1} + \p{2} +  (a+1) \p{1 \xrightarrow{a} 2}$$
If, for ease of notation, we write a map $f \in B^\Sigma = B^{\lbrace a, b \rbrace}$ as a pair \\
$(a \mapsto f(a), b \mapsto f(b))$ we have 
$$\begin{array}{llll}
\alpha(\lambda(\rho(a))) \ &= \ ((a \mapsto 1, &b \mapsto 0), \ &a) \\
\alpha(\lambda(\rho(b))) \ &= \ ((a \mapsto 0, &b \mapsto 1), \ &b) \\
\alpha(\lambda(\rho(ab))) \ &= \ ((a \mapsto 1, &b \mapsto a), \ &ab) \\
\alpha(\lambda(\rho(aa))) \ &= \ ((a \mapsto a+1, &b \mapsto 0), \ &aa)
\end{array}$$
\end{example}

\begin{lemma}\label{lemma_morphismtosemidirect}
The map $\alpha$ is a monoid morphism from the image $\lambda(\rho(\Sigma^*))$ in $\ol{\S[\Sigma]} \Gamma_{2,\Sigma}$ to the semidirect
product $G$.
\end{lemma}
\begin{proof}
First, since $\lambda$ and $\rho$ are monoid morphisms, the identity element of $\lambda(\rho(\Sigma^*))$ is the identity
element of $\ol{\S[\Sigma]} \Gamma_{2,\Sigma}$, which is $\p{1} + \p{2}$. Now the definition of $\alpha$ gives 
$\alpha(\p{1} + \p{2}) = (0, 1)$ which is the identity element in $G$.

Now let $p, q \in \lambda(\rho(\Sigma^+)) \subseteq \ol{\S[\Sigma]} \Gamma_{2,\Sigma}$, say $p = \lambda(\rho(u))$ and $q = \lambda(\rho(v))$. We aim to
show that $\alpha(pq) = \alpha(p) \alpha(q)$, which by definition of $\alpha$ and the multiplication in $G$ means exactly that
$a_{pq} = a_p a_q$ in $A$ and that $b_{pq}(\sigma) = b_p(\sigma) + a_p b_q(\sigma)$ for all $\sigma \in \Sigma$.

Firstly, using the definition of multiplication in $\ol{\S[\Sigma]} \Gamma_{2,\Sigma}$ and the fact that empty paths in a quiver decompose
only as their own squares we have 
$$a_{pq} \ = \ (pq)(\p{1}) \ = \ p(\p{1}) q(\p{1}) \ = \ a_p a_q.$$

Now for each $\sigma \in \Sigma$, by the definition of $b_{pq}$ and of the multiplication in $\ol{\S[\Sigma]} \Gamma_{2,\Sigma}$ and
that fact that a path of length $1$ in a quiver decomposes only by pre- and post-multiplying with the appropriate empty paths,
$$b_{pq}(\sigma) \ = \ (pq)(\p{1 \xrightarrow{\sigma} 2}) \ = \ p(\p{1}) q(\p{1 \xrightarrow{\sigma} 2}) + p(\p{1 \xrightarrow{\sigma} 2}) q(\p{2}).$$
By Remark~\ref{remark_emptypath} we have that the coefficient of the empty path $\p{2}$ in $\rho(v)$ is a monomial in the variables $\sigma_2$ for different $\sigma$, so by the definition of $\lambda$ we have $q(\p{2}) = 1$.
Also, by definition we have that $p(\p{1}) = a_p$ and $q(1 \xrightarrow{\sigma} 2) = b_q(\sigma)$, so that
$$b_{pq}(\sigma) \ = \ a_p b_q(\sigma) + b_p(\sigma)$$
as required to complete the proof.
\end{proof}

\begin{theorem}\label{thm_freerep3}
Let $\S$ be a commutative semiring and $\Sigma$ a (not necessarily finite) alphabet. Let $G$ be the semidirect product defined above. The
submonoid [subsemigroup] generated by the elements of the form $\alpha(\lambda(\rho(\sigma)))$ for $\sigma \in \Sigma$
is a free object (on the given generating set) in the monoid variety [semigroup variety] generated by $UT_2(\S)$.
\end{theorem}
\begin{proof}
By Lemma~\ref{lemma_morphismtosemidirect}, the map $\alpha$ is a monoid morphism on $\lambda(\rho(\Sigma^*))$ so the 
submonoid [subsemigroup] of $G$
generated by the given elements is clearly the image of under $\alpha$ of $\lambda(\rho(\Sigma^*))$ [respectively, $\lambda(\rho(\Sigma^+))$].
Thus, by  
Theorem~\ref{thm_freerep2} and Lemma~\ref{lemma_morphismtosemidirect}, it suffices to show that $\alpha$ is
injective on the the set $\lambda(\rho(\Sigma^*))$. Suppose $p \in \ol{\S[X_{1,\Sigma}]} \Gamma_{n,\Sigma}$ lies in this image, and
write $\alpha(p) = (b,a) \in G$. Then:
\begin{itemize}
\item the coefficient in $p$ of each path of the form $\p{1 \xrightarrow{\sigma} 2}$ can be recovered from $\alpha(p)$ as $b(\sigma)$;
\item the coefficient in $p$ of the empty path $\p{1}$  can be recovered from $\alpha(p)$ as $a$; and
\item the coefficient in $p$ of the empty path $\p{2}$ is always $1 \in \S$. Indeed, we have $p = \lambda(\rho(w))$ for some $w \in \Sigma^+$; by Remark~\ref{remark_emptypath} the coefficient $f_{\p{2}}^w$ of $\p{2}$ in $\rho(w)$ contains only variables of the form $\sigma_2$ for $\sigma \in \Sigma$, and so the map $\lambda$ takes this coefficient to $1$.
\end{itemize}
Thus, all the coefficents in $p$, and hence $p$ itself, can be recovered from $\alpha(p)$, so $\alpha$ must be injective on $\lambda(\rho(\Sigma^+))$.
\end{proof}

\begin{remark}
In the proof of injectivity in Theorem~\ref{thm_freerep3} we used the second component $\alpha(p)$ only to recover the coefficient in $p$ of the
empty path $\p{1}$. But in fact this coefficient, and hence the entire of $p$, can be recovered from the coefficient of paths of the form $\p{1 \xrightarrow{\sigma} 2}$ and hence from the first component of $\alpha(p)$. Indeed, it follows from Remark~\ref{remark_emptypath} that the coefficient we seek is (with our new
notation identifying $\sigma_1$ with $\sigma$) simply the abelianisation of the word $w$. Now it is easy to see that the unique highest-degree term in any coefficient of a length-$1$ path in $p$ will arise in the coefficient of $\p{1 \xrightarrow{\sigma} 2}$ where $\sigma$ is the last letter of $w$, and
that this coefficient will be the abelianisation of $w$ without its last letter. Thus, we may recover the coefficient $\p{1}$ by finding the highest
degree term in any coefficient of a path of length $1$, and multiplying this coefficient by the letter labelling the path in question.

It is therefore possible in principle to define the free object on $\Sigma$ in the given variety as a certain subset of
$B^{\Sigma}$ (instead of $B^{\Sigma} \times A$) under an appropriate product operation. However, the definition of
this operation is technically involved, and it is simpler just to use both components.
\end{remark}

\begin{remark} As remarked in the introduction, it was shown in \cite[Theorem~4.1]{Daviaud18} that $UT_2(\T)$ generates the \textit{bicyclic variety}, that is,
the semigroup variety generated by the bicyclic monoid, or equivalently (as a consequence of a result of Scheiblich \cite{Scheiblich71}, as explained in
\cite[Section~7]{Daviaud18}) by the free monogenic
inverse monoid. Thus, Theorem~\ref{thm_freerep3} yields in particular a representation of each free object
in this variety.

Note that the free monogenic inverse monoid is so named because, when
equipped with a natural unary \textit{inverse} operation, it is a free object of rank $1$ in a certain variety (2,1)-algebras; it is
not free (of any rank) in any variety of semigroups or monoids. Indeed, it is easy to show that a free object in any semigroup or
monoid variety either satisfies a torsion identity (an identity of the form $A^i = A^j$ with $i \neq j$) or is generated by indecomposable
elements (that is, elements which cannot be written as a product $xy$ unless either $x$ or $y$ is an identity element). In contrast,
the free monogenic inverse monoid has elements of infinite order, and its only indecomposable element is the identity.
\end{remark}

\begin{remark}
Theorem~\ref{thm_freerep2} shows that for $\S$ an idempotent commutative semiring, every free object in the semigroup variety
generated by $UT_2(\S)$ embeds in a semidirect
product of a commutative monoid acting on a semilattice. One might wonder whether there are any identities
inherently satisfied by such semidirect products (and hence by all upper triangular matrix semigroups over
idempotent commutative semirings), but in fact such a semidirect product can even contain a free subsemigroup.

For example, let $A$ be the free commutative monoid on some alphabet $\Sigma$. For a
word $w \in \Sigma^+$ we write $\ol{w}$ for image of $w$ in $A$ (which is essentially
the abelianisation of $w$).
Let $\mathcal{P}(A)$ be the power set of $A$ considered as a semilattice under the operation of union, and
let $A$ act on $\mathcal{P}(A)$ by
$$aX = \lbrace ax \mid x \in X \rbrace$$
for all $a \in A$ and $X \subseteq A$. Let $H$ be the resulting semidirect product of $A$ and $\mathcal{P}(A)$,
and consider the subsemigroup generated by the elements of the form
$\hat{\sigma} = (\lbrace \overline{\sigma} \rbrace, \overline{\sigma})$
where $\sigma \in \Sigma$. Given a word $w = w_1\dots w_n \in \Sigma^+$ where each $w_i \in \Sigma$ it is readily verified that
$$\hat{w_1} \dots \hat{w_n} \ = \ (\lbrace \ol{w_1}, \ol{w_1 w_2}, \dots, \ol{w_1 \dots w_n} = \ol{w} \rbrace, \ \ol{w})$$
The first component is the set of abelianized prefixes of $w$. Since $w$ can clearly be reconstructed from this set, it
follows that extension of the map $\sigma \mapsto \hat{\sigma}$ to a morphism from $\Sigma^+$ to $H$ is injective, in other
words, its image is isomorphic to a free subsemigroup of rank $|\Sigma|$ inside $H$. 

Another example of this phenomenon is implicit in Example~\ref{example_freeidpt} below, which shows that
upper triangular matrices over idempotent commutative semirings need not satisfy identities.
\end{remark}

\section{Applications}\label{sec_fields}

In this final section, we consider some properties of the free objects and varieties under consideration, which can be deduced
from the descriptions in the preceding sections. First, we consider the question of when $UT_n(\S)$ satisfies any
identities at all or, at the other extreme, when the free objects in the variety it generates contain (absolutely) free subsemigroups. We show that if a commutative semiring $\S$ has the property that distinct formal polynomials in one variable represent distinct functions then $UT_2(\S)$ (and hence also $UT_n(\S)$ and $M_n(\S)$ for all $n \geq 2$) cannot satisfy a semigroup identity. Combined with results of
Okninski and Salwa \cite{Okninski95} this observation allows us to describe exactly for which fields $k$ the semigroup $UT_2(k)$ satisfies an identity or contains a free subsemigroup.

We remark that in what follows we regard a polynomial in fewer than $n$ variables as a (degenerate) example of a polynomial in $n$ variables.

\begin{lemma}\label{lemma_distinctmultivariate}
Let $\S$ be a commutative semiring and $n$ a natural number. Then distinct formal polynomials in one variable represent distinct functions over $\S$ if and only if distinct polynomials in $n$ variables represent distinct functions over $\S$.
\end{lemma}
\begin{proof}
The direct implication is proved by strong induction on $n$. The base case $n=1$ is trivial. Suppose for induction that $n \geq 1$ and that distinct formal polynomials in $n$ or fewer variable always represent distinct functions over $\S$, and let $f$ and $g$ be distinct formal polynomials in $n+1$ variables. Choose a variable $x$, and write $f$ and $g$ as $f = \sum p_i x^i$ and $g = \sum q_i x^i$ where the $p_i$s and $q_i$s are polynomials in the variables other than $x$. Since $f$ and $g$ are distinct formal polynomials there is an $i$ for which $p_i$ and $q_i$ are distinct formal polynomials, which by the
inductive hypothesis means that $p_i$ and $q_i$ define distinct functions. Choose values for the variables other than $x$ at which $p_i$ and
$q_i$ take different values. Substituting these into $f$ and $g$ we obtain distinct (since at least the coefficient of $x^i$ differs) formal polynomials in the single variable $x$. By the original assumption, we may choose a value of $x$ at which these polynomials take different values. But these values are
the values of $f$ and $g$ respectively for the chosen values of all the variables. Thus, $f$ and $g$ differ as functions over $\S$.

The converse is obvious in view of the remark above the statement of the lemma.
\end{proof}

\begin{theorem}\label{thm_distinct}
If $\S$ is a commutative semiring in which distinct one-variable polynomials define distinct functions, then $UT_2(\S)$ does not satisfy any semigroup or monoid identity (and
hence neither does $UT_n(\S)$ nor the full matrix semigroup $M_n(\S)$ satisfy any semigroup or monoid identity for $n \geq 2$).
\end{theorem}
\begin{proof}
By Lemma~\ref{lemma_distinctmultivariate} we may assume that distinct multi-variable polynomials over $\S$ define distinct functions.

Clearly if a semigroup satisfies a monoid identity $u=v$ then it satisfies the semigroup identity $\sigma u = \sigma v$ for any symbol
$\sigma$, so it suffices to show that $UT_2(\S)$ does not satisfy a semigroup identity. Suppose for a contradiction that it satisfies the
identity $u = v$ where $u,v \in \Sigma^+$. By Theorem~\ref{thm_freerep2} this will mean that $f_\pi^u = f_\pi^v$ for every path $\pi$ in, which because of the assumption means that all these equations holds for
formal polynomials. But the definitions of these formal polynomials are independent of the semiring over which we are working
(modulo a formal identification of the multiplicative identity elements of the different semirings). So if this were true then by
Theorem~\ref{thm_freerep2} we would have that the identity $u=v$ is satisfied in $UT_2(\S)$ for \textit{every} semiring $\S$.
However, it is well-known (see for example \cite{Volkov15}) that $UT_2(\mathbb{N})$ (for example) contains a free subsemigroup of rank $2$, and therefore does not
satisfy any semigroup identity.
\end{proof}

We remark in particular that Theorem~\ref{thm_distinct} applies to infinite fields, yielding the following fact which we suspect is known to
experts but which we have not managed to locate in the literature:

\begin{corollary}\label{cor_infinitefields}
If $k$ is a field then $UT_2(k)$ satisfies a semigroup identity if and only if $k$ is finite.
\end{corollary}

Theorem~\ref{thm_distinct} also allows us to show that $UT_2(\S)$ need not satisfy an identity even if $\S$ is an idempotent commutative
semiring:

\begin{example}\label{example_freeidpt}
Let $X$ be a non-empty set of symbols and let $\S = \B[X]$ be the semiring of formal polynomials in the variables $X$ over the Boolean
semifield $\B$. Then $\S$ is a commutative, idempotent semiring; in fact it is the free commutative, idempotent semiring on the set $X$.
Now let $t$ be a new symbol not in $X$ and let $p$ and $q$ be two distinct formal polynomials in $\S[t]$. We are interested in whether $p$
and $q$ can be equal in $\overline{\S[t]}$, that is, whether they could both define the same function from the formal polynomial semiring $\B[X]$
to itself.

Choose a symbol $x \in X$ and choose $i \in \mathbb{N}$ to exceed the degrees of $p$ and $q$. We claim that $p(x^i) \neq q(x^i)$ in
$\B[X]$. Indeed, since $p$ and $q$ are distinct formal polynomials there is a $j$ such that the coefficient of $t^j$ is different in $p$ and in
$q$. By swapping $p$ and $q$ if necessary, we may suppose without loss of generality that some monomial $m$ appears in this coefficient
in $p$ but not in $q$. Now it is easy to see that the formal polynomial $p(x^i)$ contains the monomial $mx^{ij}$ while $q(x^i)$ does not. 
Thus, distinct formal polynomials in one variable over $\B[X]$ define distinct functions, so by Theorem~\ref{thm_distinct}, $UT_2(\B[X])$
does not satisfy a semigroup identity.
\end{example}

Recall that a semigroup or monoid is called \textit{locally finite} if every finitely generated subsemigroup or monoid is finite. A
\textit{variety} of semigroups or monoids is called \textit{locally finite} if every finitely generated semigroup or monoid in the variety is finite
(or equivalently, if every semigroup or monoid in the variety is locally finite, or if the finite rank free objects are finite).

A topic of recurring interest is the relationship between the antithetical properties of satisfying a semigroup identity, and containing a free
subsemigroup. It is well-known and easy to see (see for example \cite{Volkov15}) that if $\S$ is a semiring of characteristic $0$ (that is,
in which the multiplicative identity $1$ generates an infinite subsemigroup under addition) then $UT_2(\S)$
contains a free subsemigroup of rank $2$ (and therefore satisfies no semigroup identities).
Okninski and Salwa \cite{Okninski95} (see also \cite[Theorem 6.11]{Okninski98}) showed (among other things) that matrix semigroups over \textit{finitely generated} fields satisfy a \textit{generalised Tits alternative}: such a semigroup either satisfies some semigroup identity or contains a free
subsemigroup of rank $2$. Combined with the above observation, this leads immediately to a complete description of fields $k$ for which $UT_2(k)$ contains a free subsemigroup:

\begin{corollary}\label{cor_infinitefields2}
If $k$ is a field then $UT_2(k)$ contains a free subsemigroup of rank $2$ if and only if $k$ is not locally finite.
\end{corollary}
\begin{proof}
If $k$ is locally finite then clearly every semigroup of matrices over $k$ is locally finite, so in particularly cannot contain a
non-trivial free subsemigroup. Conversely, if $k$ is not locally finite then we may choose an infinite, finitely generated subfield $m$ of $k$. By Corollary~\ref{cor_infinitefields}, $UT_2(m)$ satisfies
no semigroup identity, so by the generalised Tits alternative of Okninski and Salwa \cite[Theorem~6.11]{Okninski98}, $UT_2(m)$ contains a free subsemigroup of rank $2$, and hence so does $UT_2(k)$.
\end{proof}

We now apply our results about free objects to describe when the semigroup and monoid varieties generated by $UT_n(\S)$ are locally finite. We shall need some basic
preliminary results.

\begin{proposition}\label{prop_localfinitequiver}
If $\S$ is a locally finite commutative semiring and $\Gamma$ a finite quiver without loops then the quiver algebra $\S \Gamma$ is a locally finite semiring.
\end{proposition}
\begin{proof}
Suppose $\S$ is locally finite and let $T$ be a finite subset of $\S \Gamma$. Since elements of
$\S \Gamma$ have finite support, only finitely many elements of $\S$ appear as coefficients in elements
in $T$; let $U$ be the subsemiring of $\S$ generated by these elements. It
follows from the definition of addition and multiplication in
$S \Gamma$ that every element of the subsemiring generated by $T$ has coefficients in $U$; in other
words $T$ lies inside $U \Gamma$. But by the assumption $\S$ is locally finite, so $U$ is finite and
$U \Gamma$ is finite.
\end{proof}

\begin{proposition}\label{prop_locallyfinite}
Let $\S$ be a commutative semiring, and suppose that for every finite set $X$ of variables the polynomial function semiring $\ol{\S[X]}$ is locally finite. Then for every $n \geq 1$ the semigroup $UT_n(\S)$ generates a locally finite variety of semigroups and of monoids.
\end{proposition}
\begin{proof} 
It suffices to show that all the free objects of finite rank in the given varieties are finite. 
Let $\Sigma$ be a finite alphabet and $n \geq 1$, and let $X_{n,\Sigma}$ and $\Gamma_{n,\Sigma}$ be as defined in Section~\ref{sec_freeobj}. By assumption  $\ol{\S[X_{n,\Sigma}]}$ is locally finite as a semiring, so by Proposition~\ref{prop_localfinitequiver}, $\ol{\S[X_{n,\Sigma}]} \Gamma_{n,\Sigma}$ is locally
finite as a semiring, and hence also as semigroup under multiplication. But by Theorem~\ref{thm_freerep1}, the free objects of rank $|\Sigma|$
in the varieties of semigroups and monoids generated by $UT_n(\S)$ embed as (finitely generated) multiplicative subsemigroups of $\ol{\S[X_{n,\Sigma}]} \Gamma_{n,\Sigma}$, and hence are finite. 
\end{proof}

\begin{theorem}\label{thm_locallyfinite}
Let $\S$ be a commutative semiring and $n \geq 1$. The following are equivalent:
\begin{itemize}
\item[(i)] the semigroup variety (or the monoid variety) generated by $UT_n(S)$ is locally finite;
\item[(ii)] $S$ satisfies a multiplicative identity not implied by commutativity;
\item[(iii)] $S$ satisfies a multiplicative torsion identity (that is, an identity of the form $A^i = A^j$ with $i < j$);
\end{itemize}
\end{theorem}
\begin{proof}
For $(i) \implies (ii)$ we prove the contrapositive. If $\S$ does not satisfy a multiplicative identity not implied by commutativity then $UT_1(\S)$ does not satisfy any identity not implied by commutativity, so the semigroup [monoid] variety it generates (and hence also those generated by $UT_n(\S)$ for $n \geq 2$) contain the free semigroup [monoid] of rank $1$ and is not locally finite.

For $(ii) \implies (iii)$, suppose $\S$ satisfies a multiplicative identity not implied by commutativity.
Then there must be some letter which occurs a different number of times on the left and right of the identity, and it follows that by substituting different powers of a single variable $A$ for the variables in this identity, we may obtain an identity of the form $A^i =A^j$ for some $i < j$.

Finally, we show that $(iii) \implies (i)$. Suppose $\S$ satisfies the multiplicative identity $A^i = A^j$ where $i < j$. It follows that two distinct powers of the element $1+1$ are equal, from which we may deduce that $\S$ has finite characteristic, that is, the additive
subsemigroup generated by $1$ is finite, say with $p$ elements. Then every monogenic additive subsemigroup has $p$ or fewer elements.
It follows that
$\S$ itself is locally finite as a semiring. Indeed, if $T$ is a finite subset of $\S$ then every element of the semiring generated by $T$
can be written as a polynomial in the elements of $T$ with integer coefficients at most $p$ and powers less than $j$.

Now if $X$ is any finite set of variables then clearly the monomials $x^i$ and $x^j$ define the same function for each $x \in X$. Let $T$ be a finite subset
of $\overline{\S[X]}$, and let $T'$ be a set of unique representatives in $S[X]$ for he functions in $T$. Let $U$ be the set of elements in $\S$
which occur as coefficients in formal polynomials in $T'$'; then $U$ generates a finite subsemiring $V$ of $\S$. Then it is easy to see that every function in the subsemiring of $\overline{\S[X]}$ generated by $T$ can be written as a sum of monomials, each of which has coefficient from $V$ from and no exponent greater than $j$. Since there are only finitely many such monomials, we deduce that the subsemiring generated by $T$ is
finite. This shows that $\overline{\S[X]}$ is locally finite, and the claim now follows from Proposition~\ref{prop_locallyfinite}.
\end{proof}

Implicit in the proof is that a necessary condition for $UT_n(\S)$ to generate a locally finite variety is that $\S$ itself be is locally finite as a semiring. This condition is not sufficient, as the following example shows:

\begin{example}\label{example_interval}
Let $\S = \lbrace -\infty \rbrace \cup [0,1]$ where $[0,1]$ is the closed interval in $\mathbb{Q}$ (or in $\mathbb{R}$, which gives a different
semiring with the same properties for this purpose), with operations defined by
$$x \oplus y = \max(x,y) \textrm{ and } x \otimes y = \min(x+y, 1).$$
Then $\S$ is a commutative, idempotent semiring with zero element $-\infty$ and identity element $0$. It is easy to see that
$\S$ is locally finite (and it follows easily that the semigroups $UT_n(\S)$ are locally finite for all $n \geq 1$). However, $\S$ does
not satisfy a multiplicative torsion identity (the identity $A^i = A^j$ with $i < j$ being falsified by for example $A = 1/j$), so by
Theorem~\ref{thm_locallyfinite} the variety generated by $UT_n(\S)$ is not locally finite for any $n \geq 1$.
\end{example}

\section*{Acknowledgements}
The author thanks the organisers and participants of the 56th Summer School on General Algebra and Ordered Sets (held in \v{S}pindler\r{u}v Ml\'{y}n, Czechia in September 2018), questions and conversations from which prompted this line of research. He also thanks Marianne Johnson for many helpful conversations.

\bibliographystyle{plain}

\def\cprime{$'$} \def\cprime{$'$}

\end{document}